\documentclass{amsart}

\usepackage{amsmath}
\usepackage{amsfonts,txfonts}
\usepackage{amsthm}
\usepackage{graphicx} 
\usepackage{amssymb,mathrsfs}
\usepackage{url}
\usepackage{color}

\newtheorem{theorem}{Theorem}[section]
\newtheorem{lemma}[theorem]{Lemma}
\newtheorem{corollary}[theorem]{Corollary}
\newtheorem{proposition}[theorem]{Proposition}

\theoremstyle{definition}
\newtheorem{definition}[theorem]{Definition}
\newtheorem{algorithm}[theorem]{Algorithm}

\theoremstyle{remark}

\numberwithin{equation}{section}

\newcommand{\abs}[1]{\lvert#1\rvert}

\newcommand{\cS}{\mathcal{S}}
\newcommand{\cN}{\mathcal{N}}

\newcommand{\bN}{\mathbb{bN}}
\newcommand{\bR}{\mathbb{R}}
\newcommand{\bQ}{\mathbb{Q}}
\newcommand{\bC}{\mathbb{C}}
\newcommand{\bZ}{\mathbb{Z}}

\newcommand{\sC}{\mathscr{C}}
\newcommand{\sH}{\mathscr{H}}
\newcommand{\sB}{\mathscr{B}}
\newcommand{\sA}{\mathscr{A}}
\newcommand{\sP}{\mathscr{P}}

\newcommand{\CC}{\mathbb{C}}
\newcommand{\CTwo}{\mathbb{C}^2}
\newcommand{\Ct}{\mathbb{C}^2}
\newcommand{\Cn}{\mathbb{C}^n}
\newcommand{\RR}{\mathbb{R}}
\newcommand{\ep}{\varepsilon}

\newcommand{\cD}{\mathcal{D}}

\newcommand{\cB}{\mathcal{B}}
\newcommand{\cR}{\mathcal{R}}

\newcommand{\sector}{\rho}
\newcommand{\Set}{\Xi}

\renewcommand{\Re}{\text{Re}}
\renewcommand{\Im}{\text{Im}}

\newcommand{\norm}[1]{\left\| #1\right\|}
\newcommand{\onorm}[1]{\abs{#1}}
\newcommand{\snorm}[1]{\norm{#1}}
\newcommand{\enorm}[1]{\norm{#1}_e}
\newcommand{\sonorm}[1]{\onorm{#1}}

\newcommand{\Henon}{H\'{e}non }

\begin{document}

\title[Computability properties of hyperbolic complex H\'{e}non maps]
{Computability properties of hyperbolic complex H\'{e}non maps}

 
\author[S. Boyd]{Suzanne Boyd}
\address{Department of Mathematical Sciences\\
University of Wisconsin Milwaukee\\
PO Box 413\\
Milwaukee, WI 53201, 
USA}
\email{sboyd@uwm.edu, ORCID: 0000-0002-9480-4848}

\author[C. Wolf]{Christian Wolf}
\address{Department of Mathematics and Statistics\\
Mississippi State University\\
Starkville, MS 39759, USA}
\email{cwolf@math.msstate.edu, ORCID: 0000-0002-7976-3574.}
\thanks{C.W.\ was partially supported by a grant from the Simons Foundation (SFI-MPS-TSM-00013897). A significant portion of the paper was developed when S.B.\  was visiting  the Department of Mathematics and Statistics of Mississippi State University. We thank the department for their hospitality and their support of the visit.}

\date{\today}

\begin{abstract}
In this article, we provide the first theoretical framework guaranteeing
that computers can, in principle, be used to analyze the parameter space of complex \Henon maps. More precisely, we obtain computability results for hyperbolic polynomial diffeomorphisms of $\mathbb{C}^2$, for which H\'{e}non maps are prototypical examples.

Specifically, we establish computability of the Julia set for hyperbolic maps, semi-decidability of hyperbolicity, and lower computability of the hyperbolicity locus in the parameter space of generalized H\'{e}non mappings of fixed degree at least two.

Our approach builds upon techniques developed in Boyd and Wolf's recent previous works on polynomial maps of $\mathbb{C}$ and polynomial skew products of $\mathbb{C}^2$. In the setting of polynomial diffeomorphisms of $\Ct$, however, establishing hyperbolicity for the Julia set is considerably more difficult, as it requires identifying unstable (and stable) cone fields that are preserved and expanded by $Df$ (respectively $Df^{-1}$), and also due to the lack of algorithmically detectable quantitative shadowing.

\end{abstract}

\maketitle

\section{Introduction}
\label{sec:Intro}

\subsection{Motivation}
Computational methods play an important role in the study of dynamical systems, both as tools for experimentation and as a method to guide theoretical investigation. This naturally raises the question of which dynamically defined objects, such as invariant sets, invariant measures, or parameter space loci, can, in principle, be computed with rigorously controlled  accuracy. These issues are studied in \textit{computability in dynamical systems}, a synthesis of dynamical systems with abstract computability theory.

The central idea of computability theory is to represent mathematical objects (such as points, sets, and functions) by convergent sequences generated by a Turing machine (which, for our purposes, can be viewed as a computer algorithm). A point, set, or function is said to be \textit{computable} if there exists a Turing machine that can produce approximations to it with any prescribed level of accuracy. For further details, see subsection \ref{sec:compute:basic} and \cite{Tu, Wei2000}.

Computability in dynamical systems has received considerable attention over the past two decades. In particular, substantial progress has been made in understanding the computability and non-computability of Julia sets in one-dimensional complex dynamics (see, e.g., \cite{BBRY-2011, BBY1, BY,  Braverman2005, Braverman-Yampolsky-2008, BY2, D1, DY2018, R2005}). More recently, researchers have also investigated the computability of a range of dynamical invariants, including entropy, general invariant sets, spectra, topological pressure, zero-temperature limits, and equilibrium states (see, e.g., \cite{Binder-et-al-2024, BDWY2022,  BurrWolf2018, Burr-Wolf-2024, BSW2020, GHR2011, GraccaEtAl2018, HP2025, HS, HM, HR}).

One of the most striking discoveries in the computability theory of one-dimensional complex dynamics is that there exist polynomials with computable coefficients whose Julia sets are nevertheless non-computable \cite{BY}. Even more remarkably, the measure of maximal entropy, the Brolin–Lyubich measure, which is supported on the Julia set, is always computable \cite{BBRY-2011}. This dichotomy highlights a fundamental distinction between the computability of sets and that of their associated invariant measures.

Much less is known in higher-dimensional complex dynamics. Polynomial diffeomorphisms of $\mathbb{C}^2$ exhibit both expanding and contracting behavior, and there are no simple algorithmic criteria that distinguish different dynamical behavior. As a result, basic computability questions in this setting have remained largely open.

Recently, Boyd and Wolf  (\cite{BoydWolf-Skew1}) studied computability of invariant sets for certain polynomial endomorphisms of $\mathbb{C}^2$, namely polynomial skew products. There, the skew product structure gives invariance of vertical complex lines, simplifying the analysis. In the present paper, we turn to invertible systems and consider polynomial diffeomorphisms of $\mathbb{C}^2$, which are up to conjugacy finite compositions of generalized complex H\'enon maps, see \cite{FM}.

The goal of this paper is to establish rigorous computability results for hyperbolic polynomial diffeomorphisms of $\mathbb{C}^2$ and to understand how hyperbolicity can be detected algorithmically. In particular, we provide effective procedures for computing Julia sets, verifying hyperbolicity, and identifying hyperbolic regions in parameter space.

In the next subsection, we describe our results in more detail.

\subsection{Statement of the results}

In this paper, we address computability questions for hyperbolic polynomial diffeomorphisms of $\bC^2$ with non-trivial dynamics. A prototype of class of such maps are complex \Henon maps given by 
\begin{equation}\label{Henon}
  H_{a,c}(z,w) = (z^2+c-aw,z),
\end{equation}
where $a\in \bC\setminus \{0\}$ and $c\in \bC$. It is straight-forward to show that $H_{a,c}$ has constant Jacobian $a$ and that the inverse $H_{a,c}^{-1}$ (which has Jacobian $a^{-1}$) is affinely conjugate to a map of the form \eqref{Henon}. 
\Henon maps were originally introduced as diffeomorphisms of $\RR^2$  for $a,c$ being real parameters.
For example, Benedicks and Carleson established the existence of chaotic behavior
in the form of a strange attractor for certain real \Henon maps \cite{BC}.  In this paper, we treat $H_{a,c}$ as a diffeomorphism of $\Ct$, and allow $a,c$ to be complex.
Foundational work on the dynamics of the complex {\Henon } family has been carried out by Bedford and Smillie  \cite{BLS1, BLS2, BS6, BS9}, Hubbard et al. \cite{HPV, HOV1, HOV2}, and Fornaess and Sibony \cite{FS1992c}.
Still, many questions remain unanswered, see, e.g., \cite{Hen} for a collection of open problems.

For a polynomial diffeomorphism of $\bC^2$, the notion of non-trivial dynamics refers to $f$ having a dynamical degree $d(f)>1$ (see Equation~\eqref{dyndeg} for the definition), which is equivalent to $f$ having positive topological entropy.   These maps turn out to be, up to conjugacy,
finite compositions of ``generalized'' \Henon maps; that is, finite compositions of maps of the form $f(z,w) = (p(z)-aw,z),$ where $p$ is a (monic) complex polynomial of degree $d\geq 2$ (see \cite{FM}).
For generalized \Henon mappings, we may without loss of generality  assume that $f$ has Jacobian smaller or equal than $1$, because otherwise we can simply consider $f^{-1}$. 

There are several invariant sets that capture the basic features of a  polynomial diffeomorphism $f$ of $\Ct$ (of dynamical degree $>1$), analogous to the case of polynomials in $\bC$. Namely, let $K^+$ resp.\ $K^-$ be the set of points whose forward resp.\ backward orbits are bounded. Define $J^{\pm} = \partial K^{\pm}$, $J=J^+\cap J^-$ and $K=K^+\cap K^-$. Both $J$ and $K$ are bounded. The set $J$ is called the Julia set of $f$. 
We say $f$ is hyperbolic if $J$ is a uniformly hyperbolic set, that is, if there exists a continuous $Df$-invariant splitting $T_J \bC^2= E^s\oplus E^u$ such that $Df$ uniformly contracts on $E^s$ and uniformly expands on $E^u$. In this case, the non-wandering set $\Omega_f$ is the union of $J$ and finitely many (possibly zero) attracting periodic points. Moreover, the saddle points are dense in $J$. Furthermore, hyperbolicity is an open property in the parameter space, and nearby hyperbolic maps are topologically conjugate. We refer to \cite{BS1}
for more details.

Our first main result concerns the computability of the Julia set for hyperbolic polynomial diffeomorphisms of $\CTwo$.  Namely, we have the following result:

\begin{theorem} \label{thm:Jcomputable}
Let $f$ be a hyperbolic polynomial diffeomorphism of $\CTwo$ of degree $d(f)>1 $. Then $J_{f}$ is computable.
\end{theorem}
In fact, we prove an even stronger result,  showing that the Julia sets of all hyperbolic polynomial diffeomorphisms of $\bC^2$ of fixed degree can be computed by one Turing machine. We delay the precise statement until after we have formally introduced the required concepts from computability theory.
The basic idea behind the proof of Theorem \ref{thm:Jcomputable} is to approximate the chain recurrent set by a decreasing sequence of finite unions of boxes. The objective is to determine the stage in this approximation at which every box contains a point of the Julia set. This property is tested by computing all saddle periodic points up to a given, but progressively increasing, period. By also incorporating the previously obtained collections of boxes, we ensure that the procedure terminates and yields the desired approximation of the Julia set.
We note that though we mention \Henon mappings in this introduction due to their significance in the literature, we prove the above theorem for polynomial diffeomorphisms without relying on the \Henon structure. 

Unlike in the recent computability work on polynomial skew products (\cite{BoydWolf-Skew1}) or one-dimensional maps (\cite{BoydWolf-1Dim}), for polynomial diffeomorphisms of $\Ct$ we provide an algorithm to compute $J$ which doesn't require establishing hyperbolicity of the map along the way; we just prove if $f$ is hyperbolic then the algorithm halts and has computed $J$ to the desired precision. However, we provide confirmation of hyperbolicity as an additional result.

\begin{theorem} \label{thm:Hyp-semi-decideable}
A polynomial diffeomorphism $f$ of $\CTwo$ with dynamical degree $d(f)>1$ being hyperbolic is a semi-decidable property. 
\end{theorem}
More precisely, the statement of Theorem \ref{thm:Hyp-semi-decideable} means that there exists a Turing machine that, on input of an oracle for the coefficients of the complex polynomials defining $f$, 
 halts if $f$ is hyperbolic, and runs forever if $f$ is not hyperbolic. 

To establish hyperbolicity we construct (non-overlapping) unstable and stable cone fields which are preserved and expanded by the map or respectively, its inverse. 
In a more general setting, Newhouse and Palis \cite{New, NP1} showed that
an $f$-invariant set $\Lambda$ is hyperbolic for $f$ iﬀ there is a field of ``unstable'' cones in the tangent
bundle over $\Lambda$ such that $Df$ maps the cone field inside itself, and such that in some
norm, $Df$ uniformly expands the cones, and a field of ``stable'' cones preserved and expanded by $Df^{-1}$. Moreover, the cone fields need not be continuous in $x \in \Lambda$;
hence, the cone field criterion for hyperbolicity yields a natural way to study the
hyperbolic structure of a diﬀeomorphism by using a computer. In our algorithm,
we build cones which are constant  on individual boxes of collections of boxes which are subsets of a (large) box $V_R$ which is guaranteed to contain $K$.

In contrast to  \cite{BoydWolf-Skew1} and \cite{BoydWolf-1Dim}, the algorithms developed in this article do not rely on the shadowing lemma, and, in particular, no computable version of the shadowing lemma is required. Instead, our approach
relies on the fact that the saddle points are dense in $J$.  We observe that, in fact, the algorithms in this article could be adapted to derive a quantitative shadowing result; that is, to algorithmically determine the dependence of the constants associated with shadowing lemma map. To the best of our knowledge, this result has not been known for hyperbolic sets of saddle type. 

In order to provide a parameter space result, we focus on generalized \Henon mappings rather than all polynomial diffeomorphisms. 
For $d\geq 2$ consider the set of polynomials 
$\sA_d = \{ z^d + a_{d-2}z^{d-2} + \cdots + a_1 z + a_0: a_\ell \in \CC \} \cong \CC^{d-2}$, parametrized by the coefficients of the polynomial. 
Let $\sH_d(a,c)$ denote the set of polynomial diffeomorphisms of $\bC^2$ of the form $(p_c(z)-aw,z),$ where $p_c\in \sA_d$ and $a\in\bC\setminus\{0\}.$

Then, by combining Theorems \ref{thm:Jcomputable} and \ref{thm:Hyp-semi-decideable} with Algorithm 4.1 and Theorem 1.3 of \cite{BoydWolf-Skew1} adapted to our setting we obtain the following:

\begin{theorem}\label{thm:hyp-locus}
 Let $d\geq 2$. 
    The hyperbolicity locus of  generalized \Henon maps within the family $\sH_d(a,c)$  is lower semi-computable.
\end{theorem}

Lower semi-computability means that there is a Turing machine which outputs a sequence of balls whose union coincides with the set of hyperbolic maps in $\sH_d(a,c)$. This sequence is, in general, infinite, and does not provide,  at any given time, information about the size of the missing hyperbolic maps in $\sH_d(a,c)$. 

We close the paper by providing a few additional computability results that are direct consequences of our algorithms, including that disconnectivity of the Julia set is semi-decidable, as well as some examples of more general settings in which our techniques can be applied, such as regular polynomial diffeomorphisms of $\Cn$.

Now, we present the organization of the sections. Section~\ref{sec:prelim} contains preliminaries and background information, including required results from computability and hyperbolicity for polynomial diffeomorphisms and \Henon mappings. 
Section~\ref{sec:results} contains the primary results of the paper - our algorithms, and the proof of their correctness, establishing Theorem~\ref{thm:Jcomputable} in Section~\ref{sec:Jcomputable}, and Theorem~\ref{thm:Hyp-semi-decideable} (and hence Theorem~\ref{thm:hyp-locus}), in Section~\ref{sec:hypsemidecideable}. 
Section~\ref{sec:otherapps} contains the disconnectivity result (Proposition~\ref{prop:J-disconnected}), and in Section~\ref{sec:generalizations} we provide applications of our techniques to more general settings such as regular polynomial diffeomorphisms of $\Cn$.

\section{Preliminaries and background}
\label{sec:prelim}
In this section, we provide an overview of the basic definitions and properties of the material required in this paper: first on computability in general, and then on hyperbolicity and \Henon mappings.

\subsection{Computability}
\label{sec:compute:basic}

We are interested in the correctness of algorithms that compute dynamically defined subsets of $\bR^\ell$, respectively $\bC^\ell$. We briefly review here the required definitions and 
 refer to 
\cite{
BBRY-2011,
Binder-et-al-2024,
Braverman2005,
BY2009,
Burr-Wolf-2024,
BSW2020,
GHR2011} 
for more detailed discussions of computability theory.  Our approach uses closely related definitions to those in \cite{BY2009} and \cite{BSW2020}. Throughout this discussion, we use a computational model via Turing machines (computer programs for our purposes).  One can think of the set of Turing machines as a particular, countable set of functions; we denote $T(x)$ as the output of the Turing machine $T$ based on input $x$.

We start with the definition of computable points in $\ell$-dimensional Euclidean space. 

\begin{definition}
Let $\ell\in \bN$ and $x\in \bR^\ell$. An \emph{oracle} of $x$ is a function $\phi:\bN\to \bQ^\ell$ such that $\Vert \phi(n)-x\Vert < 2^{-n}$. Moreover, we say $x$ is \textit{computable}
if there is a Turing Machine $T=T(n)$ 
which is an oracle of $x$.
\end{definition} 
It is straightforward to see that rational numbers, algebraic numbers, and some transcendental numbers such as ${ e}$ and $\pi$ are computable real numbers. However, since the collection of Turing machines is countable, most points in $\bR^\ell$ are not computable.
Identifying $\bC^\ell$ with $\bR^{2\ell}$, the notion of computable points naturally extends to $\bC^\ell$.

Next, we extend the notion of computable points to points in computable metric spaces.

\begin{definition}\label{def:computable}
Let $(X,d_X)$ be a separable complete metric space with metric $d_X$, and let $\cS_X=\{s_i: i\in \bN\}\subset X$ be a countable dense subset. We say $(X,d_X,\cS_X)$ is a \textit{computable metric space} if the distance function $d_X(.,.)$ is uniformly computable on $\cS_X\times \cS_X$, that is, if there exists a Turing machine $T=T(i,j,n)$, which on input $i,j,n\in \bN$ outputs a rational number such that
$|d_X(s_i,s_j)-T(i,j,n)|<2^{-n}$.
\end{definition}
The points in $\cS_X$ in Definition \ref{def:computable} are called the \textit{ideal points} of $X$, and $\cS_X$ is the ideal set of the computable metric space. The ideal points assume the role of $\bQ^\ell$ in $\bR^\ell$. We may suppress the subscript $X$ and write $(X,d,\cS)$ instead of $(X,d_X,\cS_X)$ when the context precludes ambiguity.

\begin{definition}
Let $(X,d_X,\cS_X)$ be a computable metric space.
An {\em oracle} for $x\in X$ is a function $\phi$ such that on input $n\in \bN$, the output $\phi(n)$ is a natural number so that $d_X(x,s_{\phi(n)})<2^{-n}$.  Moreover, we say $x\in X$ is {\em computable} if there is a Turing machine $T=T(n)$ which is an oracle for $x$.
\end{definition}

It is easy to see that $(\mathbb{R}^\ell,d_{\mathbb{R}^\ell},\cS_{\mathbb{R}^\ell})$, with $d_{\mathbb{R}^\ell}$ the Euclidean distance on $\bR^\ell$ and $\cS_{\mathbb{R}^\ell}=\mathbb{Q}^\ell$, is a computable metric space.

Next we define computable  functions between computable metric spaces.

\begin{definition}\label{def:computablefunction}
Let $(X,d_X,\cS_X)$ and $(Y,d_Y,\cS_Y)$ be computable metric spaces and  $\cS_Y=\{t_i: i\in \bN\}$.  Let $D\subset X$.  A function $f:D\rightarrow Y$ is {\em computable} if there is a Turing machine $T$ such that for any  $x\in D$ and any oracle $\phi$ of $x$, the output $T(\phi,n)$ is a natural number satisfying $d_Y(t_{T(\phi,n)},f(x))< 2^{-n}$.

\end{definition}

The composition of computable functions is computable because the output of one Turing machine can be used as the input approximation for subsequent machines.  In addition, basic operations, such as the arithmetic operations and the minimum and maximum functions, are computable. See \cite{BHW2008} for more details on these topics.

\subsection*{Computability of sets, and the $L^\infty$ metric}

Next, we introduce the computability of compact subsets of computable metric spaces.
Recall that the {\em Hausdorff distance} between two {compact} subsets $A$ and $B$ of a metric space $X$ is given by
$$
d_H(A,B)=\max\left\{\max_{a\in A}d(a,B),\max_{b\in B}d(b,A)\right\},
$$
where $d(x,C)=\min\{d_X(x,y): y\in C\}$.   In other words, the Hausdorff distance is the largest distance of a point in one set to the other set. 

\begin{definition} \label{defn:2^n-approx}
    A set $\Set_n$ is called a \textit{$2^n$-approximation of $C$} if, in the Hausdorff metric, $d_H(C,\Set_n)\leq 2^{-n}$. 
\end{definition}

\noindent \textbf{Notation.}  
We use the notation $B(s,\delta)$ for the $\delta$-ball about a point $s$, and the notation $\cN(A,\delta)$
for the $\delta$-neighborhood about a set $A$.

\begin{definition}
\label{defn:S_Y}
Let $(X,d_X,\cS_X)$ be a computable metric space.
We say that a ball $B(x,r)$ is an \textit{ideal ball} if $x\in \cS_X$ and $r=2^{-i}$ for some $i \in \bZ$.

 For 
    $\sC= \{C\subset X\,\, {\rm compact}\}$ let $\cS_\sC = \cS_\sC(X)$ denote the collection of finite unions of closed ideal balls. We also denote the sets in $\cS_\sC$ as ideal sets.
\end{definition}

We will use the following well-known property, see e.g., \cite{BY2009}:

\begin{lemma}\label{lem:compact-subsets-computable-metric-space}
Let $(X,d_X,\cS_X)$ be a computable metric space, and let $\sC$ and $\cS_\sC$ be as in Definition~\ref{defn:S_Y}. 
Then 
$(\sC,d_H,S_\sC)$ is a computable metric space.
\end{lemma}

While considering the computability of sets in $\bC^\ell$ we slightly deviate from \cite{DY2018}  where the Euclidean metric is used. This is because it  is more natural for computer calculations in $\bC^\ell$
to consider vectors in $\mathbb{R}^{2\ell}$
rather than $\CC^\ell$, and use the $L^{\infty}$ metric, rather than 
Euclidean.

\medskip

\noindent \textbf{Notation.}  
When we write $\norm{\cdot}$ we mean the $L^{\infty}$ 
norm on~$\mathbb{R}^{2\ell}$, so that for a vector $x=(x_1,\ldots,x_\ell)\in \mathbb{C}^\ell,$
\begin{equation}
\label{eqn:boxnorm}
\snorm{x}=\max\{ \abs{\Re(x_j)},\abs{\Im(x_j)} \colon 1\leq j\leq \ell\}.    
\end{equation}
Hence, 
$$d(z,w) = \max \{ | \Re(z_j)-\Re(w_j)|, |\Im(z_j)-\Im(w_j)|: 1\leq j \leq \ell \}, $$ 
and if $z$ and $w$ are in a (closed) box of sidelength $r$, then $d(z,w) \leq r$. 
We may use the simpler notation $\sonorm{\cdot}$ in one complex dimension.

The $L^\infty$ metric is {\em uniformly
equivalent} to the euclidean metric on~$\mathbb{C}^\ell, \enorm{\cdot}$, because
$
\frac{1}{\sqrt{2\ell}} \enorm{x} \leq \snorm{x} \leq \enorm{x}.
$
Neighborhoods are slightly different concerning two uniformly equivalent
norms, but the topology generated by them is the same; thus, they
can practically be used interchangeably.

We may say \textit{box} when we mean a ball around a point in the $L^\infty$ norm.  Analogously to \cite{DY2018}, we consider ideal balls to have dyadic rational side length and dyadic rational center coordinates. Indeed,
for $\ell\in \mathbb{Z}^+,$ we consider the computable metric space $(\bC^\ell,d,\cS)$, where $d$ is the $L^\infty$ metric
and $\cS$ is the set of points $x=(x_1,\ldots,x_\ell)$ with the real and imaginary parts of each of $x_i$ dyadic rationals; i.e., points in $\cD = \{ a/2^b : a\in\bZ, b\in \bN\}$.

From now on, we primarily focus on $\CTwo$ as a domain of our dynamical system, though we consider a higher-dimensional $\CC^\ell$ as a parameter space. 

Similar to \cite{Braverman-Yampolsky-2008} aside from using the $L^\infty$ metric,
applying Definition~\ref{defn:S_Y} and Lemma~\ref{lem:compact-subsets-computable-metric-space} to $\bC^2$ we conclude:

\begin{corollary} \label{cor:set-computable}
A compact set $C \subset \bC^2$ is \text{computable} if there is a Turing Machine $T(n)$ which on input $n \in \bN$, outputs an ideal set $\Set_n$ which satisfies satisfies $d_H(C,\Set_n)\leq 2^{-n};$ i.e., 
$\Set_n$ is a \text{$2^n$-approximation of $C$}.    
\end{corollary}

As mentioned before, in our setting, the Hausdorff metric is based on the $L^\infty$ metric. 

\smallskip

Finally, to study the locus of maps for which an interesting invariant set has some stable behavior (in our case, the locus of hyperbolicity within the family of generalized \Henon mappings of a fixed degree), we use the following. 

\begin{definition}
\label{defn:LowerComputableSet}
Let $D\subset \CC^{\ell}$ be open. We say $D$ is {\em lower semi-computable} if there exists a Turing machine $T=T(n)$ which on input $n\in \bN$ outputs $(x_n,r_n)$, where $x_n \in \CC^\ell$ has dyadic rational coordinates, $r_n>0$ is a  dyadic radius, and $D = \cup_{n=1}^\infty B(x_n,r_n)$. 
\end{definition}

Note that in the definition of lower semi-computability, we do not require an error estimate for how close any finite union of dyadic balls is to the set $D$, just that it converges in the limit.

\subsection{H\'{e}non dynamics and invariant sets}

Polynomial diffeomorphisms of $\CTwo$ have polynomial inverses,
so are often called polynomial automorphisms.
Friedland and Milnor (\cite{FM}) showed that
polynomial automorphisms of $\CTwo$ fall in two categories.
\textit{Elementary} automorphisms have simple dynamics, and are
polynomially conjugate to a diffeomorphism of the form $(z,w) \mapsto
(az+b, cw+p(z))$ ($p$ polynomial, $a,c \neq 0$). \textit{Nonelementary}
automorphisms are conjugate to
finite compositions of \textit{generalized \Henon
mappings}, of the form $f(z,w) = (p(z)-aw,z)$, where $p(z)$ is a
monic polynomial of degree $d>1$ and $a \neq 0$.

To clarify the situation, one can define a \textit{dynamical degree} of a 
polynomial automorphism of $\CTwo$.  If \textit{deg}$(f)$ is the maximum of the 
degrees of the coordinate functions, the {\it dynamical degree} is 
\begin{equation}\label{dyndeg}
d = d(f) = \lim_{n \to \infty} \textit{deg}(f^{n})^{1/n}. 
\end{equation}
This degree is a conjugacy invariant.  Elementary automorphisms have 
dynamical degree $d=1$.  A non-elementary automorphism is conjugate to some
automorphism whose polynomial degree is equal to its dynamical degree. 
Without loss of generality, we may assume that $f$ is a finite
compositions of generalized \Henon mappings, rather
than merely conjugate to mappings of this form.

Thus, the quadratic, complex \Henon family $H_{a,c} (z,w)=(z^2+c-aw,z)$
represents the dynamical behavior of the simplest class of nonelementary
polynomial automorphisms; those of dynamical degree two.  In this paper, we usually
use the letter $f$ for a polynomial diffeomorphism of $\CTwo$ with
$d(f)>1$, and $H$ for a (often degree two) \Henon mapping.

\medskip

\noindent \textbf{The chain recurrent set.}
The \textit{chain recurrent set}, $\cR(f)$, is $\cR=\cap_{\alpha>0} \cR(\alpha)$, where $\cR(\alpha)$ is the set of all $\alpha$-recurrent points $(x,y)$ of $f$, that is, the set of points admitting an $\alpha$-pseudo orbit starting and ending at $(x,y)$. 

Bedford and Smillie (\cite{BS2})  showed that for $f$ a polynomial 
diffeomorphism of $\Ct$ with $d(f)>1$,  $f$ is hyperbolic  on its Julia set
$J$
iff $f$ is hyperbolic on its chain recurrent set $\mathcal R$ iff
$f$ is hyperbolic on its nonwandering set $\Omega$.
Thus we say $f$ is \textit{hyperbolic} if any of these conditions holds.
In fact, in \cite{BS1} Bedford and Smillie show that if $f$ is hyperbolic, then
$\mathcal R$ and  $\Omega$ are both equal to $J$ union
 finitely many attracting periodic  orbits; and $J=J^*,$ where $J^*$ is defined as the closure of the saddle periodic points.
Thus for hyperbolic polynomial diffeomorphisms of $\Ct$, 
the 
basic sets are $J$ and the attracting periodic orbits.

\medskip

\noindent \textbf{Hyperbolicity.}
Let $f$ be a diffeomorphism of $\Ct$.  If $p$ is a periodic point of 
period $m$, and the eigenvalues
$\lambda, \mu$ of $D_pf^m$ satisfy $\abs{\lambda} > 1 > \abs{\mu}$ (or
vice-versa), then $p$ is a {\em saddle periodic point}.
The eigenvalue with the larger (smaller) norm 
is called the unstable
(stable) eigenvalue.

In the following, let $f$ be a diffeomorphism of a manifold $M$, and let $\Lambda\subset M$ be a compact, $f$-invariant set.  First we recall the standard definition of hyperbolicity (see, e.g., \cite{Rob}):
%
$\Lambda$
is {\em hyperbolic} for $f$ if at each $x$ in $\Lambda$, there is a splitting of the tangent bundle $T_xM
= E^s_x \oplus E^u_x$, which varies
continuously with $x \in \Lambda$, such that:
\begin{enumerate}
 \item $f$ preserves the splitting, \textit{i.e.}, $D_xf(E^s_x) = 
E^s_{fx}$, and $D_xf(E^u_x) = E^u_{fx}$, and 
\item $Df \ (Df^{-1})$ is uniformly expanding on $E^u \ (E^s)$, \textit{i.e.}, there exists a constant $L > 1$ and a norm $\norm{\cdot}_x$ on $T_{\Lambda} M$, continuous for $x \in \Lambda$, for which
\begin{eqnarray*}
&   \text{ if } \mathbf{w} \in E^u_x, \text{ then } 
 \norm{D_x f(\mathbf{w})}_{f(x)} 
 \geq 
 L \norm{\mathbf{w}}_x, 
 \text{ and }\\
& \text{ if } \mathbf{w} \in E^s_x, \text{ then } 
 \norm{D_x f^{-1}(\mathbf{w})}_{f^{-1}(x)} 
 \geq 
 L \norm{\mathbf{w}}_x.
\end{eqnarray*}
\end{enumerate}

%
As noted in the introduction, Newhouse and Palis (\cite{New, NP1}) show hyperbolicity can be described using a {\em cone field}.  To define unstable cones, $C^{u}_x$ at each $x$ in $\Lambda$, we need a
splitting $T_x M = E^u_x \oplus E^s_{x}$, and a
positive real-valued function $\sector_{u}(x)$ on~$M$.
Let $\mathbf{e}^{u/s}_x \in E^{u/s}_x$ be the unit vector in $E^{u/s}_x$, and
 define the
$\sector_{u}(x)$-sector $S_{\sector_u(x)}(E^{u}_{x}, E^{s}_{x})$ by
\begin{equation}    
\label{eqn:cone-defn}
 S_{\sector_u}(x)(E^{u}_{x}, E^{s}_{x})  = 
\{ \mathbf{v} = \mathbf{v}^{u}_x + \mathbf{v}^{s}_x
\in E^{u}_{x} \oplus E^{s}_{x}:
      \norm{\mathbf{v}^s_x} \leq \sector_u(x)\norm{\mathbf{v}^u_x} 
      \},
\end{equation}
where $\mathbf{v}_x^{u/s}
=v^{u/s}_x \mathbf{e}^{u/s}_x 
=\text{Proj}_{E^{u/s}_x}(\mathbf{v})$ for a $\textbf{v}^{u/s}_x \in \CC$ is the projection onto $E^{u/s}_{x}$ of the vector $\mathbf{v} \in T_xM$, 
so
$\mathbf{v} 
= \mathbf{v}^u_x + \mathbf{v}^s_x
= v^u_x \mathbf{e}^u_x + v^s_x \mathbf{e}^s_x 
\in E^u_{x} \oplus E^s_{x}$. 
Thus, 
$\mathbf{v}\in C^u_x$ means 
$|v^s_x| \leq \sector_u(x) |v^u_x|$. 
Here instead of the $L^\infty$-metric, we use the Euclidean norm on the tangent space. 
Then $ S_{\sector_u(x)}(E^u_{x}, E^s_{x}) $
is the sector of vectors ``closer'' to $E^u_x$, as measured w.r.t.\ the weighting $\sector_u(x)$, and 
$C^u_x := S_{\sector_u(x)}$.  We define stable cones $C^s_x$ entirely analogously, keeping the same splitting as was used to define the unstable cones, otherwise swapping $u$'s and $s$'s in the above.

Newhouse and Palis show that $\Lambda$  
is hyperbolic for $f$ iff there is a field of cones $\{ C^{u(s)}_x \subset T_xM \colon x \in \Lambda\}$, 
a constant $L > 1$, and a continuous norm 
$\norm{\cdot}$, such that
 $Df\ (Df^{-1})$ preserves the unstable (stable) cones, \text{i.e.}, $D_x f(C^u_x) \subset C^u_{f x}
 \ ( D_x f^{-1}(C^s_x) \subset C^s_{f^{-1} x}$), and such that in this norm,
$D_x f \ (D_x f^{-1})$ uniformly expands vectors in $C^u_x  \ (C^s_x)$, by at least $L$;
moreover, the field of cones are required to vary continuously with $x$.
(In their proof (\cite{NP1}), they first show that the existence of a
cone field preserved by $Df$ implies the existence of a continuous
splitting preserved by $f$, with the unstable (stable) directions lying
inside the unstable (stable)  cones.)

Since we are covering $\cR$ by a collection of boxes,
we aim to define a piecewise constant conefield such that each $C^u_k\ (C^s_k)$ (unstable(stable) cone) is constant on a box $B_k$,
and $Df\ (Df^{-1})$ preserves and expands the unstable(stable) cones. We have the stable cones $C^s_k$ must lie within the complements of the unstable cones: $C^s_k \subset  \CTwo\setminus C^u_k$. 

In this paper, we use both directions of hyperbolic cone fields. In order to show that the algorithm to detect hyperbolicity halts if the map is hyperbolic (semi-decideability), we have to show that if the map is hyperbolic, then the algorithm will detect it, by finding the piecewise constant, preserved, expanded cones via our algorithm. Conversely, it's a quick application of Newhouse/Palis, if we establish preserved and expanded cone-fields w/ the expansion, to conclude the map is hyperbolic.

\medskip

Since $f$ is a polynomial diffeomorphism, 
recall (as detailed in \cite{Hruska-HenonHyp} for \Henon mappings) we can calculate
a bounding box $V_R =[-R,R]^4$ containing all of recurrent behavior (namely, $\cR$), so that in our algorithm of this article, and all of our box collections lie in $V_R$.  

In order to control some of the error related to discretization of the algorithm, we use the following. 

\begin{lemma} \label{lem:calculus}
Given $f_{a,c}$ a polynomial diffeomorphism
of $\Ct$ of dynamical degree $d(f) >1$, 
let $R>0$ be such that $\cR(f)\subset V_R = [-R,R]^4.$
Let $\sB$ be a finite collection of ideal closed boxes whose union is contained in $V_R$. 
Then for each positive integer $m$, we can compute a (dyadic rational) number $Q$  (i.e., there is a Turing machine $T=T(a,c,\cB,m)$ which outputs a positive (dyadic) rational number 
$Q$) satisfying
    $$
\norm{
\frac{D_x f^m (\mathbf{v})}
{\norm{D_x f^m (\mathbf{v})}} 
- \frac{D_y f^m (\mathbf{v})}{\norm{D_y f^m (\mathbf{v})}} }
\leq Q \norm{x-y},
$$
for any two points $x,y\in V_R$ which lie together in the same box $B \in \sB$, and any \text{unit} vector $\mathbf{v}\in T_x\Ct$.
\end{lemma}

\begin{proof}
First, in $V_R$, by calculating second derivatives we can calculate a bound $Q>0$ (which depends on the parameters defining $f$ and the iterate $m$, as well as $R$) such that:
$$
Q \geq \sup_{y\in V_R}\left\{\norm{D^2_{y} f^m }\right\}.
$$

Now using that bound, 
for each box $B$ of $\sB$, if $x,y\in V_R$, we can use a multi-dimensional Mean Value Theorem to conclude the desired inequality.
\end{proof}

\section{Results}
\label{sec:results}

In this section, we provide the proofs of our main results. In subsection~\ref{sec:Jcomputable}, we establish Theorem~\ref{thm:Jcomputable} on computability of $J$  by providing Algorithm~\ref{alg:computeJ} which computes a $2^N$-approximation of $J$ for any given $N$, and then proving (in Proposition~\ref{prop:J-generalized-henon-computable}) that it halts and is correct (i.e., that it does indeed produce a $2^N$-approximation of $J$).
In subsection~\ref{sec:hypsemidecideable}, we establish
Theorem~\ref{thm:Hyp-semi-decideable} on semi-decideability of hyperbolicity---which implies Theorem~\ref{thm:hyp-locus}, lower computability of the hyperbolicity locus in a parameterized family. In subsection~\ref{sec:otherapps} 
(Proposition~\ref{prop:J-disconnected}) we establish the semi-decidability of hyperbolic Julia sets being disconnected.

\subsection{The Julia set of a hyperbolic \Henon map is computable}
\label{sec:Jcomputable}

In the following we consider a polynomial diffeomorphism $f(z,w)=(p(z,w),q(z,w))$, of dynamical degree $d(f)>1$ given by complex polynomials $p(z,w)$ and $q(z,w)$. Algorithmically these polynomials are given by oracles of the coefficients of $p$ and $q$.

To establish Theorem~\ref{thm:Jcomputable},
we show the following:

\begin{proposition}
   \label{prop:J-generalized-henon-computable}
There is a Turing Machine $T=T(N,p,q)$ which on input of $N\in \bN$, and oracles of the coefficients of  $p, q:\Ct\to\CC$ defining the polynomial diffeomorphism $f(z,w)=(p(z,w),q(z,w))$ (of dynamical degree $d(f)>1$), has the following property:
If $f$ is hyperbolic then $T$ halts and outputs  ideal sets $\Set_N$ and  $\Set'_N$ in $\Ct$ such that $\Set_N$ is a $2^{-N}$-approximation of  $\cR(f)$, and $\Set'_N$ is a $2^{-N}$-approximation of  $J$. 
\end{proposition}

Theorem~\ref{thm:Jcomputable} is an immediate consequence of Proposition \ref{prop:J-generalized-henon-computable} and the definition of the computability of a compact set (see Corollary~\ref{cor:set-computable}).

Our strategy, which begins the same as Boyd and Wolf's \cite{BoydWolf-Skew1, BoydWolf-1Dim}, is to approach the Julia set by examining the chain recurrent set, $\cR$. We produce a collection of boxes covering $\cR$, along with information on the map $f$ acting on those boxes. The corresponding algorithm uses ideas from the work of Boyd (under her former last name, Hruska: \cite{Hruska-HenonHyp, Hruska-HenonJ}), which had as its goal establishing hyperbolicity of some \Henon mappings rigorously using computer program. One of the principal differences between the results of this article and that prior work is that the main theorem of \cite{Hruska-HenonHyp} is: if the algorithm reports the map is hyperbolic, then it is. But, there is no guarantee that the algorithms will succeed if the map is hyperbolic. 
In this paper, we start with some tools of \cite{Hruska-HenonHyp, Hruska-HenonJ} as a foundation, but provide a different algorithm for which we prove that, if the underlying map is hyperbolic, the algorithm will  compute $J$; additionally, we provide an algorithm that will establish hyperbolicity of the map.
Formally, we produce the following. 

\begin{definition} \label{defn:boxchrecmodel}
Let $\cR$ be the chain recurrent set of a map $g \colon \mathbb{C}^\ell \to \mathbb{C}^\ell$.
Let $\Gamma = (\mathcal{V}, \mathcal{E})$ be a directed graph,
with vertex set $\mathcal{V} = \{ B_k \}_{k=1}^m$,  where the sets $B_k$ are closed boxes in $\mathbb{C}^\ell$ with pairwise disjoint interior, and such that the union of the boxes $\cB = \cup_{k=1}^m B_k$ contains $\cR$. 
We further require that there is an edge from $B_k$ to $B_j$ if 
$g(B_k)$ 
intersects $B_j$, that is,  
\[
 \{ (k, j) \colon 
g(B_k)
\cap B_j  \neq \emptyset \} \subset \mathcal{E}(\Gamma).
\]
We also require that  $\Gamma$ is the disjoint union of strongly connected components $\Gamma'_i$, $i=1,\dots,s$; that is, for $i\in 1,\dots,s$ and any $B_k, B_j\in \Gamma'_i$, there is a path in $\Gamma'_i$ from $B_k$ to $B_j$.  

If these properties hold
we say that $(\cB, \Gamma)$ is a {\em box chain recurrent model} of $g$ on $\cR$, and the components $(\cB_i, \Gamma_i)$ are the {\em box chain components} of the model.
\end{definition}

A box chain recurrent model provides an approximation of the dynamics of the map $g$ on $\cR$, 
and the strongly connected components provide approximations of the chain components.  

To establish Proposition~\ref{prop:J-generalized-henon-computable}, we first provide the algorithm which, on input $N,p,q$, produces the ideal sets $\Set_N$ and $\Set'_N$. After providing the algorithm, we prove that if the map is hyperbolic, then the algorithm halts and produces the ideal sets which are  $2^N$-approximations of $\cR,J$, respectively.

\begin{algorithm}[\textbf{Computing $2^N$-approximations of $\cR$ and $J$}]
\label{alg:computeJ}
The basic structure of the algorithm is a loop, through increasing positive integers $n$, in which we  calculate box chain components to deeper levels $\cB_n$,  until we find a large enough $n$ that a halting condition (related to $N$) is reached.
At that point we use $\cB_n$ to define $\Set_N$.

Rather than starting the loop at $n=0$, it turns out we need to start at $n'$ the smallest integer satisfying $2^{n'} > \max \{2^{N+2} R, \frac{2}{R}\}$. (We see why in the proof of the Proposition, following the algorithm). So, the first step is to calculate $R$ for $f$, and this $n'$. 

Next, let $\ell_n$ be any increasing sequence of positive integers. 

Now, for $f=f(p,q)$, suppose we are in our loop at some $n\geq n'$. For this $n$, we invoke the first Algorithm provided in Section 3 of \cite{BoydWolf-Skew1} (there provided for polynomial skew products of $\Ct$, with obvious adaptations for the setting of polynomial diffeomorphisms of $\Ct$ to that algorithm and to the lemmas in Section 2 of the same article).  This produces a box chain recurrent model $(\cB_n, \Gamma_n)$, satisfying Definition~\ref{defn:boxchrecmodel} above, and yields a collection of boxes $\sB_n = \{ B_j \}$, which are a subset of the boxes formed from a $(2^n)^4$ grid on in a trapping region $V(f)=[-R,R]^4$, such that the boxes' union $\cB_n = \cup_j B_j$ contains the chain recurrent set---including
(all of) the saddle component(s) of the chain recurrent set and all of any attracting components.

Examination of the algorithm that produces these boxes shows they are nested upon refinement: $\cB_n$ is produced from subdividing the boxes of $\cB_{n-1}$ by bisecting each side.

See Figure~\ref{fig:Hypatia} for an example of a box cover, and a refinement of that box cover, of $\cR$ for a complex \Henon mapping, generated by an implementation of the algorithm of \cite{Hruska-HenonJ}, which is the basis of the first algorithm  from Section 3 of \cite{BoydWolf-Skew1}.

\begin{figure} \centering
\includegraphics[width=.475\textwidth]{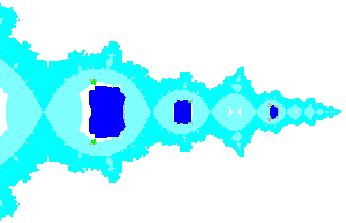}
\includegraphics[width=.475\textwidth]{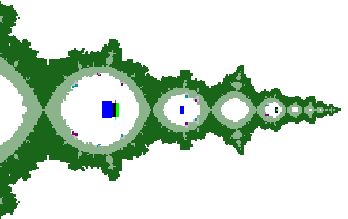}
\caption{\label{fig:Hypatia}
For $f_{a,c}$ the \Henon mapping with  $c=-1.1875, a=.15$, the chain recurrent set is $J$ and a period $2$ attracting cycle.  Shown in this figure are box chain recurrent sets, restricted to the unstable manifold of a saddle fixed point, with its natural parameterization. A heuristic algorithm shades the box chain transitive component containing $J$ in order to illustrate how close the component is to $J$. 
On the left,  boxes are of side length $2R/2^6$ and $2R/2^7$, where $R = 1.9$.  
This is  the crudest box chain recurrent model which separates $J$ from the sink in this algorithm implementation.  
On the right is a refinement obtained from subdividing once  the boxes on the left.}
\end{figure}

Next, for this $n$ we compute all periodic orbits of $f$ of period up to $\ell_n$, with precision $2^{-2n}$ (we observe that this precision is increasing with $n$). This produces a set of (ideal) points $\sP_n$.

Let $2\sB_n$ refer to the collection of boxes with the same centers as the boxes in $\sB_n$, but with double the sidelength. We let $2B_j$ denote such a box. 

Now, while holding $n$ constant, we perform a sub-loop: for each $k$ from $k=n'$ up to $k= n,$ we test whether every box in $\sB_k$ intersects $\sP_n$. 

If  the answer is yes for some $k$, then we stop the algorithm, and
 define $\Set_N:= \cup \{ 2B_j : 2B_j \in 2\sB_k\}$; 
 that is, we define this set to be the union of all of the boxes with centers the same as the boxes in $\sB_k$ (for a special $k$ with $n' \leq k \leq n$) but with double the sidelength.  

\medskip

Next, we decompose the ideal set $\Set_N$ to obtain the ideal set $\Set'_N$ being a $2^{-N}$-approximation of $J$. 

We have the graph $\Gamma_k$ underlying $\sB_k$, and we can decompose it into its strongly connected components,  $\{\Gamma_{n,i}\}$, to partition the collection of boxes $\sB_k$ into box chain transitive components $\{\sB_{k,i}\}$. 

For each box strongly connected component, $\sB_{k,i}$, simply choose any box say $B_0 \in \sB_{k,i}$. We know the double of this box contains a point of $\sP_n$, thus a periodic point of $f$. Calculate the eigenvalues of the iterate of $f$ to the period of that point, to a precision needed to distinguish the norms of the eigenvalues from $1$. (We show below that if $f$ is hyperbolic, the algorithm to calculate these eigenvalue norms away from $1$ will halt).
Once all eigenvalues of a periodic point have been calculated bounded away from 1, 
there are two cases.
If both eigenvalues have norm less than $1$,  that component is of attracting type. Otherwise, they are split with one eigenvalue having norm greater than $1$ and one eigenvalue having norm less than one, to conclude that this component is of saddle type. 

We let $\sP'_n$ denote the subset of $\sP_n$ of periodic points of saddle type, 
we let $\sB'_k$ be the collection of all box chain transitive components of saddle type, and we define 
$\Set'_N = \cup \{ 2B_j: B_j \in \sB'_k \}$; that is $\Set'_N$ is the union of the doubled boxes corresponding to the boxes in the saddle-type components.
\qed 
\end{algorithm}

We next show that this algorithm halts and produces what is claimed in the $f$ being hyperbolic case, thus establishing Proposition~\ref{prop:J-generalized-henon-computable} and hence Theorem~\ref{thm:Jcomputable}.

\begin{proof}[Proof of Proposition~\ref{prop:J-generalized-henon-computable}]
First, by choice of $n'$ s.t.\ $2^{n'} > \max \{2^{N+2} R, \frac{2}{R}\}$,
we have that for all $n\geq n'$, first $2R/2^n < 2^{-N-1}$, hence the sidelength of the boxes in $\cB_n$ satisfies
$\ep_n = 2R/2^n 
< 2^{-N-1}$, 
and  ${R}/{2^{n+1}} > 2^{-2n}.$ 

Thus, the sidelength of the boxes $2B_j$ in $2\sB_n$ is $4R/2^n < 2^{-N},$ since $n\geq n'$.

We claim if for a $k \in \{n',\ldots,n\}$ we have stopped the above algorithm, because every box $\sB_k$ intersects $\sP_n$, then each $2B_j$ in $2\sB_k$ (the set of doubled boxes) contains a periodic point of period at most $\ell_n$.

For, the precision of our approximate periodic points in $\sP_n$ was $2^{-2n}$. This means if $\sP_n$ intersects a box $B_j$ of $\sB_n$,  there is definitely a periodic point (of period at most $\ell_n$) within $2^{-2n}$ of $B_j$. Now the 
radius of $2B_j$ is $R/2^n$ and the radius of $B_j$ is $R/2^{n+1}$, and these boxes have the same center point. 
Thus the distance from the boundary of $B_j$ to the boundary of $2B_j$ is $R/2^n - R/2^{n+1} = R/2^{n+1} > 2^{-2n}=2^{-n-n},$ since $R/2 > 2^{-n} $ for all $n\geq n'$ by choice of $n'$.

If  the algorithm has halted at $n$, then we just established that this collection of (doubled) boxes each contains (for sure) a periodic point. 
Observe by choice of $n'$, we also know that these (doubled) boxes also have sidelength bounded by $2^{-N}$. That means that every point in $\Set_N$  lies in a box $2B_j$ of sidelength $< 2^{-N}$ and this box contains a periodic point. Thus, since the periodic points lie in $\cR$, we know $\Set_N \subset \cN (\cR, 2^{-N})$.

But also, by construction of the boxes $\sB_k$, we know that $\cR$ lies in the union of the boxes: $\cR \subset \cB_k = \cup\{B_j: B_j \in \sB_k\}$. Thus, $\cR$ also lies in the union of the doubled boxes, which is $\Set_N$. 

Hence, $\cR \subset \Set_N \subset \cN (\cR, 2^{-N}),$
and thus $d_H(\cR, \Set_N) < 2^{-N}$. This shows that $\Set_N$ is a $2^{N}$-approximation of $\cR(f).$

As for $\Set_N'$, 
we claim that since $f$ is hyperbolic, we can successfully determine whether the ``hyperbolic type'' of each component is saddle or attracting. 
Indeed, since $f$ is (uniformly) hyperbolic, each periodic point is either saddle or attracting. So, we can increase precision until the norms of the eigenvalues of the finite collection of saddle periodic points and attracting periodic points in $\sP_n$ are bounded away from $1$, so that the algorithm above can continue.

Finally, we must guarantee halting; that is, if $f$ is hyperbolic, then for some $n\geq n'$ and some $k\leq n$ we find every box in $\sB_k$ contains a point of $\sP_n$. This is true since we are calculating the periodic points to increasing period $\ell_n \uparrow \infty$ and increasing precision $2^{-2n}$, and because for the hyperbolic map $f$, the Julia set $J$ is the closure of the saddle periodic points, and the chain recurrent set is $J$ union with finitely many attracting cycles. 
\end{proof}

It follows from the construction above that the ``doubles" of the remaining boxes in $\sB_k \setminus \sB'_k$ are an ideal set $2^{-N}$-approximation of the attracting periodic cycles.

\subsection{Semi-decideability of hyperbolicity for 
polynomial diffeomorphisms of $\Ct$.}
\label{sec:hypsemidecideable}

 We established in the previous subsection that the Julia set of a hyperbolic diffeomorphism of $\Ct$ is computable. The corresponding algorithm does not include a feature to actually detect hyperbolicity, that is the topic of this subsection. Detecting  hyperbolicity of the map is the key ingredient of the proof Theorem~\ref{thm:Hyp-semi-decideable}. 
Specifically, we provide an algorithm which proves  the following. 

\begin{proposition}\label{prop:proof-hypalg}
There is a Turing Machine $T=T(p,q)$, which on input of   oracles of the coefficients of  $p, q:\Ct\to\CC$ defining a polynomial diffeomorphism $f(z,w)=(p(z,w),q(z,w))$ with dynamical degree $d(f)>1$, has the following property.

If $f$ is hyperbolic, then $T$ halts and confirms the hyperbolicity of $f$. 
    Moreover, $T$ outputs $N, m\in\bN$, an ideal set $\Set_N'$ in $\Ct$ which is a $2^{N}$-approximation of $J$, $\lambda>0$,
    and a field of disjoint invariant  unstable (stable) cones $C^u_j \ (C^s_j)$ in the tangent space to each box, one pair of cones per box in $\Set'_N$,
    such that on $\Set'_N,$ the derivative $Df^m \ (Df^{-m})$ is expanding in the unstable (stable) cones by at least $1+\lambda/4$.
\end{proposition}

As before, the initial step in our approach is Algorithm~\ref{alg:computeJ}. 

\begin{algorithm}[\textbf{Detecting hyperbolicity for a hyperbolic polynomial diffeomorphism of $\Ct$}]
\label{alg:hyp-detect}

Consider a  polynomial diffeomorphism $f(z,w)=(p(z,w),q(z,w))$, of dynamical degree $d(f)>1$, defined by oracles of the coefficients of the polynomials $p, q : \Ct \to \CC$.
The following algorithm loops through an increasing sequence of positive integers $N$.

\underline{\textbf{Step 1}}: In the first step, for given $N\in \bN$ and given oracles of the coefficients defining  $f$ we apply Algorithm~\ref{alg:computeJ}, which calculates an $R>0$ (which is determined only by $f$), which sets the bounding box $V_R$ used throughout the remaining steps. The input $N$ determines an approximation accuracy $2^N$,  
which then determines a smallest integer $n'$ satisfying $2^{n'} \geq \max( 2/R, 2^{N+2}R)$. 
The integer $n'$ is the starting parameter in the loop of Algorithm~\ref{alg:computeJ}, which we run through an increasing sequence of integers $n\geq n'$ until we have (guaranteed for a hyperbolic map) determined an $n\geq n'$ and a $k$ in $\{n',\ldots,n\}$ and a collection of boxes $\sB'_k$ (each lying in $V_R$) such that the union of the boxes in $2\cB'_k$ (where recall the 2 indicates ``doubles'', boxes with the same centers but double the sidelength of $\cB'_k$) is a $2^N$-approximation of $J$, and such that each box $2B \in 2\sB'_k$ is guaranteed to contain a saddle periodic point of period at most $\ell_n$ (where $\ell_n$ is  an increasing sequence calculated in the initial steps of Algorithm~\ref{alg:computeJ}), and we have produced a set of points $\sP'_n$ of approximations to these saddle periodic points in $2\cB'_k$ of period up to $\ell_n$, calculated to precision $2^{-2n}$. We recall that $k$ and $n$ depend on $N$.

\underline{\textbf{Step 2}}: 
After completing Algorithm~\ref{alg:computeJ} for $N$,
 we compute $m\in\bN$ and  $\lambda>0$ such that
 \begin{equation}\label{eqn:define-m-lambda-unstable}
\norm{D_x f^m(\mathbf{e}^u_x)} \geq 1 + \lambda
\end{equation}
holds for all saddle periodic points $x\in \sP'_n$ and all $\mathbf{e}^u_x\in E^u_x$ with $\norm{\mathbf{e}^u_x}=1$. Without loss of generality we require \eqref{eqn:define-m-lambda-unstable} also to hold for all points on the (finite) orbits of $x\in\sP'_n$.
We conclude that $\norm{D_x f^m\vert_{E^u_x}} \geq 1 + \lambda$ holds for all those points.
We note that this computation can be accomplished by simultaneously testing increasing iterates $m$ and decreasing expansion rates $\lambda$  and at the same time computing $\norm{D_x f^m(\mathbf{e}^u_x)}$ and $ \norm{D_xf^m|_{E^u_x}}$ at increasing precision. This computation halts  if $f$ is hyperbolic because in this case \eqref{eqn:define-m-lambda-unstable}  holds for some $m\in\bN$ and some $\lambda>0$ for all $x\in J$. By performing the analogous computation for $f^{-1}$ and $E^s$ (also making sure that $m$ and $\lambda$ work for $f$ and $f^{-1}$) we can assure that
\begin{equation}
    \label{eqn:define-m-lambda-stable}
\norm{D_x f^{-m}(\mathbf{e}^s_x)} \geq 1 + \lambda,
\text{ hence } \norm{D_xf^{-m}\vert_{E^s_x}}\geq 1 + \lambda
\end{equation}
for all $x\in \sP'_n$ and all $\mathbf{e}^s_x\in E^s_x$ with $\norm{\mathbf{e}^s_x}=1$.

By executing this algorithm with increasing $N$, we compute a sequence of increasing iterates $m_N$ and decreasing expansion constants $\lambda_N$. In the following, for readability, for a specific $N$ we do not include the subscript on  $m$ and $\lambda$. 

\underline{\textbf{Step 3}}: 
For the determined $m$ and $\lambda$ (depending on $N$),  we define a weighting constant 
\begin{equation}
\label{eqn:defn-sector}
\sector = \sector_\lambda :=\frac{1+\lambda}{1+ \frac{\lambda}{2}}-1    
\end{equation}
which will define the unstable (hence stable) cones in each box, see equation~\eqref{eqn:cone-defn}.
That is,
for each box $B_j$ in $\sB'_k$ in turn, consider an $x\in \sP'_n \cap B_j$. (If $B_j$ contains more than one point of $\sP'_n$, we simply choose one.)
Since $x$ is a saddle periodic point,  we may compute unit basis vectors in $E^u_x$ and $ E^s_x$ at any desired precision. We need to compute the basis vectors at a   larger precision than $2^{-2n}$, and then in the remaining estimates, make the constants slightly larger or smaller as appropriate to guarantee the required precision.
We define the unstable cone $C^u_x$ (this will be the unstable cone for the box $2B_j$ associated with $x$, which we also may denote by $C^u_j$), as 
\begin{equation}
\label{eqn:our-unstable-cone-defn}
    C^u_x = S_{\sector_\lambda}(E^u_x, E^s_x) = 
\{ \mathbf{v} 
\in E^u_{x} \oplus E^s_{x}:
      \norm{\mathbf{v}^s_x} \leq \sector_\lambda \norm{\mathbf{v}^u_x} 
      \},
\end{equation}
where $\mathbf{v}^{u(s)}_x$ is
the projection of a vector $\mathbf{v}\in T_x\Ct$ onto $E^{u(s)}_x.$  It turns out this $\sector_\lambda$ is precisely what is needed to guarantee that expansion by $1+\lambda$ on the stable and unstable spaces translates into the slightly weaker expansion by $1+\lambda/2$ in the cones.

Since we are using a common $m,\lambda$ for the stable iterate and expansion bound, the stable cone $C^s_x$ is defined analogously, and with the same weighting function $\sector_\lambda$:
\begin{equation}
\label{eqn:our-stable-cone-defn}
    C^s_x = S_{\sector_\lambda}(E^s_x, E^u_x) = 
\{ \mathbf{w} 
\in E^u_{x} \oplus E^s_{x}:
      \norm{\mathbf{w}^u_x} \leq \sector_\lambda \norm{\mathbf{w}^s_x} 
      \}.
\end{equation}

Note we must have $C^s_x$ inside of the complement of the unstable cone $C^u_x$. A straightforward calculation shows, since $\lambda>0$, we have 
$\sector_\lambda \in (0,1)$, with $\sector_\lambda\to 1$ as $\lambda\to \infty$ and $\sector_\lambda\to 0$ as $\lambda \to 0$. Since $\sector_\lambda < 1$, 
the stable and unstable cones at a point $x$ do not intersect except at $\mathbf{0}$.

We will show in the proof of Proposition~\ref{prop:proof-hypalg} below  that defining the cones in this way guarantees strict invariance of the unstable(stable) cones under $Df^{m}$ $(Df^{-m})$ (e.g., for every cone $C^u_x$ we have $D_xf^m(C^u_x) \subsetneq C^u_{f^m(x)} $ and analogously for $C^s_x$ under $Df^{-m}$), as well as expansion by at least $1+\lambda/2$ in the unstable (stable) cones, under $Df^m \ (Df^{-m})$.

\medskip
\underline{\textbf{Step 4}}: 
Since the required properties of the computed cones work at the finite collection of saddle periodic points, the next task is to calculate an upper bound for the box size, so that for every point in each box, the same cone works. That is, we claim that if $f$ is hyperbolic, then if the boxes are sufficiently small, for any $y$ in the same box as an $x\in \sP'_n$, we can use the same definition for $C^u_y$ as for the cone $C^{u(s)}_x$ (i.e., we still have preservation and expansion under $D_y f^m$ of $C^u_y$ if it's defined based on the eigenspaces at $x$ and the same $\sector_\lambda$). 

To find a suitable upper bound for our box size, we must quantify how to measure (and thus keep  away from $0$) by some normalized distance between the image under $D_xf^m$ of the boundary of the unstable cone at $x$, and the boundary of the unstable cone at $f^m(x)$. 
We do this as follows. For each $x\in \sP'_n$, we chose a unit vector $\mathbf{b}_x \in T_x\Ct$   in $\partial C^u_x$. Hence $\mathbf{b}_{f^m(x)} \in T_{f^m(x)}\Ct$ is a unit vector in the boundary $\partial C^u_{f^m(x)}.$ 
We want to find a positive, lower bound (uniform over all $x\in \sP'_n$) for the distance between the unit vector $\mathbf{b}_{f^m(x)}$ in $\partial C^u_{f^m(x)}$  and an appropriate unit vector in
$D_xf^m(\partial C^u_x)$, which is also the line of  $D_x f^m(\mathbf{b}_x)$. In fact, since $D_xf^m$ is expanding in the unstable cones, and $\mathbf{b}_x$ is a unit vector, we know $\norm{D_xf^m(\mathbf{b}_x)}>1,$ which shows that this vector has a larger norm than the unit vector(s) on its same complex line,
see Figure~\ref{fig:cone-perturb}.

\begin{figure}[h] \centering
\includegraphics[width=\textwidth]{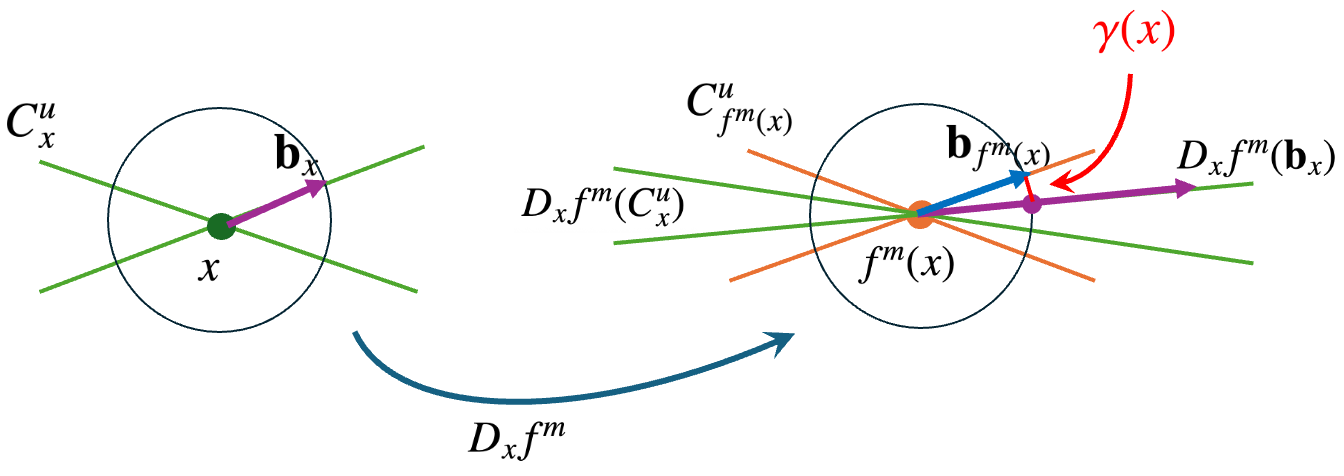}
\caption{\label{fig:cone-perturb}
\textit{Left}: An illustration of an unstable cone $C^u_x$ based at a point $x$, and a unit vector $\mathbf{b}_x$ in the cone's boundary $\partial C^u_x$. As these objects are all complex, the illustration is in the real case.
$D_xf^m$ maps the left figure onto the right.
\\ \textit{Right}: $C^u_{f^m(x)}$, the unstable cone at $f^m(x)$, is shown containing the image under $D_x{f^m}$ of the figure on the left: the images of the cone $C^u_x$ and the vector $\mathbf{b}_x$. Additionally, on the right figure we indicate the unit vectors $\mathbf{b}_{f^m(x)}$ in $\partial C^u_{f^m(x)}$ and $\gamma(x)$ the distance between this  $\mathbf{b}_{f^m(x)}$ and a unit vector in a line 
 $D_x{f^m}(\partial C^u_x)$ (which lies in the same complex line as $D_xf^m (\mathbf{b}_x)$). }
\end{figure}
Thus, for each $x\in \sP'_n$, we calculate 
the distance:
\begin{equation}
    \label{eqn:defn-gamma-x}
\gamma(x) := \norm{ 
\frac{D_xf^m(\mathbf{b}_x)}{\norm{D_xf^m(\mathbf{b}_x)}} 
- \mathbf{b}_{f^m(x)}
},
\end{equation}
where of course the resulting $\gamma$ depends not only on $x$ but also $m$.

Because the unstable cones are strictly preserved by $Df^m$, we conclude that $\gamma(x)>0$. This is because for each $\mathbf{b}_x$ which lies in the (closed) cone $C^u_x$, its image under $D_xf^m$ lies strictly inside of $C^u_{f^m(x)},$ thus so does the unit vector on the same line, which in turn is thus a positive distance from a unit vector 
$\mathbf{b}_{f^m(x)}$ in the boundary of $C^u_{f^m(x)}.$

Since $D_xf^{m}\vert_{E^{u/s}_x}$ is a linear map on a complex line, it suffices to compute the distance between two unit vectors to normalize this distance calculation.
Next, we can define 
\begin{equation}
    \label{eqn:defn-gamma}
\gamma:= \min\{ \gamma(x): x\in \sP'_n\},
\end{equation}
and note that $\gamma>0$ holds since the set of $x$'s is finite, and each $\gamma(x)>0$. As before, $\gamma$ depends on and $m$ and  $f$.

\underline{\textbf{Step 5}}: 
Next, we apply Lemma~\ref{lem:calculus} to $f$, the calculated $V_R$ (which is independent of $m$), the finite set of $x\in \sP'_n$ of saddle periodic points, each in a box $2B_j$ in 
 our collection of boxes $2\sB'_k$ (such that the union of the boxes in $2\sB'_k$ is a $2^N$-approximation of $J$),
and our selected unit vectors $\mathbf{b}_{x}\in T_x\Ct$,
so that for our positive integer $m=m_N$, we  compute a (dyadic rational) $Q$ satisfying
\begin{equation}
\label{eqn:defn-Q}
\norm{
\frac{D_x f^m (\mathbf{b}_x) }{\norm{D_x f^m (\mathbf{b}_x)} }
-\frac{D_y f^m (\mathbf{b}_x)}{\norm{D_y f^m (\mathbf{b}_x)}} 
}
\leq Q \norm{x-y},    
\end{equation}
for any $y$ in the box $2B_j$ containing $x$. 
We then compute a (dyadic rational) $\Delta>0$ (which  depends on $m$, $\lambda$ and $f$)
satisfying
\begin{equation}
\label{eqn:defn-Delta}
\Delta < \frac{1}{Q} \min \left\{\gamma,\frac{\lambda}{4}\right\}.
\end{equation}
Here, we calculate  $\Delta$  as large as possible satisfying \eqref{eqn:defn-Delta} and at a sufficient precision to compute this step in the algorithm.

We will show in the proof of Proposition \ref{prop:proof-hypalg} that
 if $y$ is in a $\Delta$-neighborhood of $x$, for some $x\in \sP'_n$, and for the unstable (stable) cones $C^{u(s)}_y$ defined identical as $C^{u(s)}_x$, 
 then the following two properties hold (i) $D_yf^m \ (D_y f^{-m}) $ preserves the cone fields (which uses $Q\Delta < \gamma)$), and  (ii) $D_yf^m \ (D_y f^{-m})$ expands by at least $1+\lambda/4$ all vectors in the cones $C^{u(s)}_y$ 
(which uses $Q\Delta < \lambda/4$). If $y$ is within $\Delta$ of more than one $x\in \sP'_n$, then any choice of $C^{u(s)}_y$ suffices.

\medskip

\underline{\textbf{Step 6}}: 
Given $x\in \sP'_n$ and $y$  in the same box as x, that is $y\in 2B_j\in 2\sB'_n$, 
we have
$\norm{x-y} < 2^{-N}$.
Since there is a  saddle periodic point in each doubled box and since the cones will be the same on the doubled boxes, we require $2^{-N} < \Delta$.
Therefore, at the integer $N$ used in  the loop, we have determined the iterate $m$ and the expansion constant $\lambda>0$ to get  the preservation and strict expansion by $Df^m$ of the unstable cones (similarly stable under $Df^{-m}$). This $m,\lambda$ in turn determined $\Delta$.
At this point, we algorithmically check if
\begin{equation}\label{eqwer1}
2^{-N} < \Delta
\end{equation}
holds by computing $\Delta$ at precision $2^{-2N}$. If yes, we halt the algorithm and report that $f$ is hyperbolic. 
It follows from the construction that  the cones  $C^{u(s)}_y$ defined the same as $C^u_x$ if $y$ is in the same (doubled) box as $x\in \sP'_n$, the derivative $Df^m_y$ preserves the cone field for all $y\in 2\sB'_n$, which in return proves that $f$ is hyperbolic. 
On other hand, if we are not able to conclude $\Delta > 2^{-N}$ from the approximation of $\Delta$, we move  to the next $N$, and start the main loop at Step 1, to refine and work with smaller boxes. 
\qed
\end{algorithm}

We will show in the proof of Proposition~\ref{prop:proof-hypalg} below that if $f$ is hyperbolic, then algorithm \ref{alg:hyp-detect} halts for some integer $N$, thus establishing that $f$ is hyperbolic. This is because as $N$ increases,  the box size decreases, and the iterate $m$, the expansion lower bound $\lambda$, and the box size bound $\Delta$ stabilize, so even though $m_N$ increases with $N$ and $\lambda_N$ decreases with $N$, and $\Delta$ depends on $m$ and $\lambda$, for sufficiently small box size the algorithm will eventually find an $N$ for which  $2^{-N} < \Delta$ holds, in which case the algorithm halts, and establishes hyperbolicity. 
     
For readability,  the proof of Proposition~\ref{prop:J-generalized-henon-computable} is split up into  several lemmas which provide justifications and/or rationales for the calculations and/or decisions in  Algorithm \ref{alg:hyp-detect}.

\begin{lemma} \label{lem:cone-pres}
    Let $f$ be a hyperbolic polynomial diffeomorphism of $\Ct$ of dynamical degree $d(f)>1$ given by oracles of the coefficients of $f$.
    Then, the cone fields defined in  equations~\eqref{eqn:our-unstable-cone-defn} and ~\eqref{eqn:our-stable-cone-defn} of Step 3 in
Algorithm \ref{alg:hyp-detect} guarantee strict invariance of the unstable(stable) cones under $Df^{m}\  (Df^{-m})$, i.e., for all $x\in \sP'_n$  we have $D_xf^m(C^u_x) \subsetneq C^u_{f^m(x)} $, and the analogous invariance for $C^s_x$ under $D_xf^{-m}$.
\end{lemma}

\begin{proof}
Recall that at this stage in the algorithm, we had determined $m$ and $\lambda$ such that
for every $x\in \sP'_n$ (our saddle periodic points with period smaller or equal than $\ell_n$),  we have that $D_xf^m\ (D_xf^{-m})$ expands unit vectors in the unstable(stable) eigenspaces of $T_x\Ct$ by at least $(1+\lambda)$, see equation~\eqref{eqn:define-m-lambda-unstable}: 
$\norm{D_xf^m(\mathbf{e}^u_x)} \geq 1 + \lambda$,
and equation~\eqref{eqn:define-m-lambda-stable}: 
$\norm{D_x f^{-m}(\mathbf{e}^s_x)} \geq 1 + \lambda.$

To show preservation of the unstable cones, consider $\mathbf{v}\in C^u_x$. 
Then $\norm{\mathbf{v}^s_x} \leq \sector_\lambda  \norm{\mathbf{v}^u_x}$.
%
%
We need to show that $D_xf^m(\mathbf{v}) \in (C^u_{f(x)})^\circ,$ so we need to show that
\begin{equation} \label{eqn:projection}
\norm{\text{Proj}_{E^s_{f^m(x)}}(D_xf^m(\mathbf{v}))} < \sector_\lambda \norm{\text{Proj}_{E^u_{f^m(x)}}(D_xf^m(\mathbf{v}))}.    
\end{equation}

We first look at the right hand side. 
Since $D_xf^m$ is linear, 
$D_xf^m(\mathbf{v}) 
= D_xf^m(v^u_x \mathbf{e}^u_x + v^s_x \mathbf{e}^s_x) 
= 
v^u_x D_xf^m(\mathbf{e}^u_x) + v^s_x D_xf^m(\mathbf{e}^s_x)$. 
But since $x$ is a saddle periodic point, we have 
$D_xf^m(\mathbf{e}^{u(s)}_x)$ lies on $E^{u(s)}_{f^m(x)}$. By the expansion guaranteed in equation~\eqref{eqn:define-m-lambda-unstable},
we have that the length of $D_xf^m(\mathbf{e}^u_x)$
(which again lies in $E^u_{f^m(x)}$) 
is at least $1+\lambda$, hence
$$|v^u_x| (1+\lambda) \leq \norm{\text{Proj}_{E^u_{f^m(x)}}(D_xf^m(\mathbf{v}))}.$$
%
Since $\mathbf{v}\in C^u_x$, it follows that
$ |v^s_x| \leq \sector_\lambda |v^u_x|,$
and we conclude that
$$
|v^s_x| (1+\lambda) \leq \sector_\lambda |v^u_x| (1+\lambda)
\leq \sector_\lambda\norm{\text{Proj}_{E^u_{f^m(x)}}(D_xf^m(\mathbf{v}))}.
$$
Next,  we consider the left hand side of the desired inequality of equation~\eqref{eqn:projection}. 
%
Observe that $D_xf^m(\mathbf{e}^{s}_x)$ lies in $E^{s}_{f^m(x)}$,
We conclude that the norm of $D_{f^m(x)} f^{-m}(\mathbf{e}^s_{f^m(x)})$ is at least $1+\lambda$. This shows that the norm of $D_x{f^m (\mathbf{e^s}_x) }$ must be at most $1/(1+\lambda)$, and this vector lies in $E^s_{f^m(x)}$.
Thus
$$
\norm{\text{Proj}_{E^s_{f^m(x)}}(D_xf^m(\mathbf{v}))} \leq | v^s_x| / (1+\lambda).
$$

Thus  $|v^s_x|/(1+\lambda) < |v^s_x |(1+\lambda)$ would imply equation~\eqref{eqn:projection}. Equivalently, it is sufficient
that $1 < (1+\lambda)^2$, which holds since $\lambda>0$.
Thus, we have established (strict) preservation of the unstable cones. 
The proof for the preservation of the stable cones  is entirely analogous.
\end{proof}

\begin{lemma} \label{lem:cone-expansion}
 Let $f$ be a hyperbolic polynomial diffeomorphism of $\Ct$ of dynamical degree $d(f)>1$ given by oracles of the coefficients of $f$.
     Then, the cone fields defined in  equations~\eqref{eqn:our-unstable-cone-defn} and ~\eqref{eqn:our-stable-cone-defn} of Step 3 in
Algorithm \ref{alg:hyp-detect} guarantee
 expansion  within the unstable(stable) cones under $Df^{m}\  (Df^{-m})$ by at least $(1+\lambda/2)$;
i.e., for all $x\in \sP'_n$  and all $\mathbf{v}\in C^u_x$ we have
\begin{equation}
  \label{eqn:unstable-cone-expansion}  
 \norm{D_xf^m(\mathbf{v})} \geq (1+\lambda/2) \norm{\mathbf{v}}.
 \end{equation}
 \end{lemma}

\begin{proof}
    As before we make use of
$
D_xf^m(\mathbf{v}) 
= 
v^u_x D_xf^m(\mathbf{e}^u_x) + v^s_x D_xf^m(\mathbf{e}^s_x)$ 
to provide a lower bound on the norm of  $D_xf^m(\mathbf{v})$. We have
\begin{equation}\label{asd1}
\norm{D_xf^m(\mathbf{v}) }
= 
\norm{v^u_x D_xf^m(\mathbf{e}^u_x) 
+ v^s_x D_xf^m(\mathbf{e}^s_x)}
\geq \norm{v^u_x D_xf^m(\mathbf{e}^u_x)} 
\geq |v^u_x |(1+\lambda).
\end{equation}
By definition $\mathbf{v}\in C^u_x$ implies $|v^s_x| \leq \sector_\lambda |v^u_x|$, which provides an upper bound on the right hand side of equation~\eqref{eqn:unstable-cone-expansion}. We have  %
$\sector_\lambda = (1+\lambda)/(1+\lambda/2) - 1$ which yields
$(1+\sector_\lambda) = (1+\lambda)/(1+\lambda/2).$
Thus,
$$
\left(1+\frac{\lambda}{2}\right) \norm {\mathbf{v}} 
\leq \left(1+\frac{\lambda}{2}\right) ( |v^u_x| + |v^s_x| )
\leq \left(1+\frac{\lambda}{2}\right)  (1+\sector_\lambda) |v^u_x|
= (1+\lambda)|v^u_x|.
$$
Together with equation~\eqref{asd1} this inequality establishes equation~\eqref{eqn:unstable-cone-expansion}. Thus, we have established expansion of $Df^m$ on the unstable cones by at least $1+\lambda/2$.
The proof for the stable cones is entirely analogous.
\end{proof}

\begin{lemma}
\label{lem:cone-pres-y}
Let $f$ be a hyperbolic polynomial diffeomorphism of $\Ct$ of dynamical degree $d(f)>1$ given by oracles of the coefficients of $f$.
   Then, during Step 5 of Algorithm~\ref{alg:hyp-detect}, 
    for all $x\in \sP'_n$ and and $y\in B(x,\Delta)$, we have that $D_yf^m \ (D_y f^{-m}) $ preserves the unstable (stable) cones $C^{u(s)}_y$.
\end{lemma}

\begin{proof}
    Let $\norm{x-y} < \Delta$. It follows from equations~\eqref{eqn:defn-Delta}, \eqref{eqn:defn-Q}, \eqref{eqn:defn-gamma-x}, and~\eqref{eqn:defn-gamma} that
\begin{equation} \label{eqn:perturb-inside-cones}
  \norm{
\frac{D_x f^m (\mathbf{b}_x) }{\norm{D_x f^m (\mathbf{b}_x)} }
-\frac{D_y f^m (\mathbf{b}_x)}{\norm{D_y f^m (\mathbf{b}_x)}} 
}
\leq Q 
 \Delta < \gamma \leq \gamma(x)
=\norm{
\frac{D_xf^m(\mathbf{b}_x)}{\norm{ D_xf^m(\mathbf{b}_x)}} - \mathbf{b}_{f^m(x)}},
\end{equation}
where $\gamma(x)$ is the Euclidean distance between the unit vector $\mathbf{b}_{f^m(x)}$ in the boundary of $C^u_{f^m(x)}$ and an (appropriate) unit vector in the complex line of $D_xf^m(\mathbf{b}_x) \in D_xf^m(\partial C^u_x) \subsetneq C^u_{f^m(x)}.$ Here we also use that $Df^m$ strictly preserves the unstable cones.

Since $\norm{x-y}<\Delta$, equation~\eqref{eqn:perturb-inside-cones} (refering back to Figure~\ref{fig:cone-perturb}) shows that the image of the boundary of the cone $C^u_x$ not just only under $D_x f^m$ but also $D_y f^m$ 
must lie inside of $C^u_{f^m(x)},$
since the image of $C^u_x$ under $D_xf^m$ does and $\gamma$ is a bound on the distance between the boundary of $C^u_{f^m(x)}$ and the image of $D_xf^m(C^u_x)$ measured for unit vectors.
We conclude that $Df^m$ preserves the unstable cone field. Here we also make use of the linearity of $D_yf^m$, which assures that it is sufficient to check the invariance  for  one set of unit vectors.  The proof for the preservation of the stable cone field by $D_yf^{-m}$ is entirely analogous.
\end{proof}

\begin{lemma}
    \label{lem:cone-expansion-y}
    Let $f$ be a hyperbolic polynomial diffeomorphism of $\Ct$ of dynamical degree $d(f)>1$ given by oracles of the coefficients of $f$.
Then the computations in Step 5 of Algorithm~\ref{alg:hyp-detect},
imply for all
 $x\in \sP'_n$ and all $y\in B(x,\Delta)$ we have that
 $D_yf^m \ (D_yf^{-m})$  expands tangent vectors in the unstable (stable) cones $C^{u(s)}_y$ by at least $1+\lambda/4$.
\end{lemma}

\begin{proof}
    Let  $\mathbf{u}\in C^u_y=C^u_x$ be a unit vector.
By first applying Lemma~\ref{lem:calculus} and then equation~\eqref{eqn:defn-Delta} we obtain
$$
\norm{D_xf^m(\mathbf{u})-D_yf^m(\mathbf{u})} 
\leq  Q\Delta \leq \lambda/4.
$$
Combining this with the fact that $D_xf^m$ expands tangent vectors in the unstable cone $C^u_x$ by at least a factor of $1+\lambda/2$ (see equation~\eqref{eqn:unstable-cone-expansion}), and that for $y\in B(x,\Delta)$  the cone $C^u_y$ is defined as $C^u_x$ (with cone axes $E^{u(s)}_x$ and width defined by the same $\sector_\lambda$), we conclude that
$$
\norm{D_yf^m(\mathbf{u})} = 
\norm{D_yf^m(\mathbf{u})-D_xf^m(\mathbf{u})+D_yf^m(\mathbf{u})} 
$$
$$
\geq 
\norm{D_yf^m(\mathbf{u})} -\norm{D_xf^m(\mathbf{u})-D_yf^m(\mathbf{u})} 
\geq (1+\lambda/2) - Q\Delta \geq 1+\lambda/4.
$$
Analogously to the cone preservation argument, the linearity of $D_yf^m$ implies that by showing expansion for any unit tangent vector by $1+\lambda/4$, the same expansion holds for all  tangent vectors, i.e., 
$\norm{D_yf^m(\mathbf{v})} \geq (1+\lambda/4) \norm{\mathbf{v}}$
for all $\mathbf{v}\in C^u_y$.
The proof for showing that tangent vectors in $C^s_y$ are expanded by $D_yf^{-m}$ by at most $1+\lambda/4$ is entirely analogous. 
\end{proof}

We now assemble the statements of the lemmas to establish Proposition \ref{prop:J-generalized-henon-computable}.

\begin{proof}[Proof of Proposition~\ref{prop:proof-hypalg}]
We discuss the steps of Algorithm~\ref{alg:hyp-detect} to show that each step is needed for correctness and halting.

First,
we already stated in Step 2 of Algorithm~\ref{alg:hyp-detect} the  justification that for a given $N$, after executing Algorithm~\ref{alg:computeJ}, the loop through iterates $m$ will halt if $f$ is hyperbolic and produces $m,\lambda$ satisfying equations~\eqref{eqn:define-m-lambda-unstable} and ~\eqref{eqn:define-m-lambda-stable}.

Second, invoking Lemmas~\ref{lem:cone-pres} and~\ref{lem:cone-expansion} establishes preservation, and expansion by a factor of $(1+\lambda/2)$, of the cones $C^{u(s)}_x$ under $Df^m \ (Df^{-m})$, for $x\in \sP'_n$. 



Third, Lemma~\ref{lem:cone-pres-y} shows that for the constants defined in equations~\eqref{eqn:defn-Delta} (specifically, $Q\Delta < \gamma$), \eqref{eqn:defn-Q}, \eqref{eqn:defn-gamma}, and~\eqref{eqn:defn-gamma-x}, we get that for all $x\in \sP'_n$ and all  $y\in B(x,\Delta)$, $D_yf^m \ (D_y f^{-m}) $ preserves the unstable (stable) cones $C^{u(s)}_x$
for all $y\in B(x,\Delta)$, which was claimed in Step 5.

Fourth, Lemma~\ref{lem:cone-expansion-y} establishes that $D_yf^m (D_yf^{-m})$ not only preserves, but expands tangent vectors in the cones $C^u_y (C^s_y)$ by at least $1+\lambda/4$, provided $\norm{x-y}<\Delta$.

\medskip

Finally, we show that if $f$ is hyperbolic
then the algorithm  halts during Step 6 a some large enough $N\geq 1$, that is, for some $N\in\bN$ the algorithm determines $2^{-N}<\Delta.$ Recall that the definition of $f$  being hyperbolic is that $J$ is a uniformly hyperbolic set of $f$, see \cite{BS1}. This means that there exists a continuous $Df$-invariant splitting $T_J\Ct=E^s\oplus E^u$ and an iterate $m'$ and a constant $\lambda'>0$ such that $Df^{m'} \ (Df^{-m'})$  expands on $E^u (E^s)$ in the Euclidean metric by at least $1+\lambda'$. In the case of saddle periodic points the splitting coincides with the splitting given by the stable and unstable eigenspaces. We note that (for the same reasons as in the Newhouse/Palis proof of the existence of cone fields), as the box sizes decrease, the iterates $m$ increase and the expansion constants $\lambda>0$ decrease. Therefore, once the algorithm computes the expansion of $D_xf^m\vert_{E^u_x}$ at a sufficiently high precision,  $m$ stops increasing when $m>m'$ is reached, and $\lambda$ stops decreasing once $\lambda<\lambda'$ holds. This shows that when the algorithm  reaches a sufficiently small enough box size,  the constructed cone fields will be a close enough approximation of the splitting, in which case the algorithm halts. 
\end{proof}

Next, we show that Proposition~\ref{prop:proof-hypalg}  implies Theorem~\ref{thm:hyp-locus}.

\begin{proof}[Proof of Theorem~\ref{thm:hyp-locus}]
The proof of the analogous result for polynomial skew products in \cite{BoydWolf-Skew1} (see Section 4) can be easily adapted to the setting of hyperbolic \Henon maps. That proof requires sweeping through an increasingly fine grid of parameters on an increasing sized box in parameter space, and at each step, for each tested parameter, rather than attempting to apply Algorithm~\ref{alg:hyp-detect} to completion (as it will run forever if the map is not hyperbolic), we run the loop only up to a maximum $N$, where $N$ increases through the main loop/ parameter space grid search of the algorithm. At each step, any map we identify being hyperbolic we  refine the same ideas of Algorithm~\ref{alg:hyp-detect} to calculate a neighborhood of this map in which all maps are hyperbolic. In other words, we allow the coefficients of the map to vary slightly while calculating error bounds. This procedure provides an open subset containing hyperbolic maps. We then continue  the search algorithm by increasing the precision of the parameter grid and simultaneously increasing $N$. Analogously as shown in \cite{BoydWolf-Skew1}, the obtained collection of open sets grows as one executes the main loop, to  cover in the limit the parameter space of all hyperbolic maps. 
\end{proof}

\subsection{Detecting disconnectivity for hyperbolic Julia sets}
\label{sec:otherapps}

Using the algorithms of the prior sections as a foundation, we can easily computationally extract additional dynamical information of interest for hyperbolic polynomial diffeomorphisms of $\Ct$. We provide a first example in this subsection: detecting disconnectivity of $J$ for a hyperbolic polynomial diffeomorphism of $\Ct$.

Indeed, it is a fairly straightforward consequence of Algorithm~\ref{alg:computeJ}---which computes an approximation of $J$ as
 a collection of boxes with non-overlapping interiors, formed as a subset of a grid on a large box $V_R = \{ (x,y): |x|,|y|\leq R\}$, and which is guaranteed to contain $J$---that if $J$ is disconnected we eventually (for a sufficiently high precision) detect the separation of $J$ into different connected components.

\begin{proposition} \label{prop:J-disconnected}
Having a disconnected Julia set is a semi-decidable property  in the space of hyperbolic polynomial diffeomorphisms of $\Ct$ of fixed dynamical degree $d>1$.
\end{proposition}

As before, to prove this result we show that there exists a Turing machine that, on input of an oracle of the coefficients defining $f$,
the Turing machine halts if $J$ is disconnected, and runs forever if $J$ is connected.

\begin{proof}[Proof of Proposition~\ref{prop:J-disconnected}]
Let $f$ be a polynomial diffeomorphism $f$ of $\Ct$ with $d(f)>1$.
Algorithm~\ref{alg:computeJ} provides, if $f$ hyperbolic, a collection of closed boxes whose union is a set $\Set'_N$ containing $J$, and with Hausdorff distance at most $2^{-N}$ to $J$. 

Thus, if $J$ is disconnected, for $N$ sufficiently large, $\Set'_N$ will  be disconnected. That is easy to detect, since $\Set'_N$ is  a finite union of closed boxes (specifically, the doubles of boxes in a grid). For example, starting with one box it is easy to list every other box which overlaps this box (check for grid neighbors), so throw the one box and boxes it touches into a set of boxes associated with one connected component of $\Set'_N$. Then add to the set all boxes which touch each of the boxes added in the prior step, and continue adding boxes which touch boxes you just added to the set, until we have identified all the boxes in a topologically connected subset of $\Set'_N$.  If there are any boxes in $\Set'_N$ not in this set, $J$ is disconnected. Going further, we can, starting with a box not in this first connected component, repeat the above to subdivide $\Set'_N$ into a finite number of distinct connected components, though if $J$ is disconnected it has infinitely many disjoint connected components so we can not separate ``all'' of them.
\end{proof}


\medskip
\section{Generalizations}
\label{sec:generalizations}

Because our techniques are grounded in general dynamical ideas (rather than specific complex-analytic results, for example), they can be adapted to various other classes of dynamical systems. 

For a more general hyperbolicity criteria we may consider Axiom A, but specifically, as our algorithms construct the chain recurrent set, we require that the chain recurrent set is a uniformly hyperbolic set and periodic points are dense in it. Note that the non-wandering set is always contained in the chain recurrent set. 

For one example of a generalization, in \cite{Wolf-Shafikov2003}, Shafikov and Wolf study regular polynomial automorphisms/diffeomorphisms of $\CC^n$, introduced by Sibony in \cite{Sibony1997}.  These are a natural generalization of \Henon maps to higher dimension. 

A polynomial automorphism $f$ of $\Cn$ is \textit{regular} if $I(f)\cap I(f^{-1})=\emptyset$, where $I(f)$ denotes
 the indeterminacy set of the rational extension of $f$
 to the projective space $\mathbb{P}^n$. 
 For such mappings $f$, the forward and 
 backward Julia sets $J^+, J^-$,
are defined as for \Henon maps, 
namely $J^\pm$ is the boundary of the set of points with bounded forward/backward orbits, and
$J=J^+ \cap J^-$
is called the Julia set of $f$.
In Section 3 of \cite{Wolf-Shafikov2003}, using results of Sibony, the authors construct  a filtration for regular automorphisms, which has similar properties to the standard filtration for \Henon maps.  The authors of \cite{Wolf-Shafikov2003} say that $f$ is \textit{hyperbolic} if $J$ is a hyperbolic set \textit{and} the saddle periodic points of $f$ are dense in $J$ (essentially, an Axiom A type of criteria). By Bedford and Smillie (\cite{BS1}), for dimension $2$, the second assumption follows from the first, but this is not true  in general in higher dimensions. 

Since in this setting, hyperbolic maps have dense periodic points  in $J$, and $J$ is a hyperbolic set, 
re-reading this article through the lens of this setting, one can see that the techniques carry forward to establish:

\begin{theorem}
\label{thm:polyauto-Cn}
The Julia set $J(f)$ of a regular polynomial automorphism $f$ of $\Cn$ is computable if $f$ is hyperbolic (i.e., if $J$ is a hyperbolic set and periodic points are dense in $J$). 
Moreover, hyperbolicity of a regular polynomial automorphism $f$ of $\Cn$ is a semi-decidable property. 
    \end{theorem}

Since the unstable and stable dimensions must add up to $n$, for $n>2$ there are stable respectively unstable spaces which are not one-dimensional. But this case can be treated similarly.  

\medskip

For another example, as mentioned earlier, Boyd and Wolf considered polynomial skew products of $\Ct$ in \cite{BoydWolf-Skew1}. A key technique of that approach is using shadowing to quantify the precision of the approximation to the purely expanding component of the chain recurrent set (the main Julia set), which was also their approach in \cite{BoydWolf-1Dim} for  polynomials of $\CC$. This is why in that paper, the algorithm to compute $J$ required establishing hyperbolicity in order to apply shadowing thus compute the precision of our estimate for $J$. For the saddle invariant set, \cite{BoydWolf-Skew1} employed a periodic points approach. Focusing on periodic points is the main technique of this article, which is why here we were able to establish the computability of $J$ without proving hyperbolicity; though we do this by using the cone field approach.  

We observe that the periodic point approach used in this article could be adapted to the expanding case as well, to establish the results of this article for polynomial skew products, by a different approach than  in \cite{BoydWolf-Skew1}. Moreover, this approach is sufficiently general that it could be applied to any polynomial \textit{endomorphisms} of $\Ct$, with the appropriate definition of a map being hyperbolic.

More precisely, a regular polynomial endomorphism on $\Cn$ is a map on $\Cn$, $n\geq 2$, of the form
$f=(f_1, \ldots, f_n)$, such that the $f_i$ are polynomial maps of degree $d\geq 2$ and $\hat{f}^{-1}_1(0) \cap \ldots \cap \hat{f}^{-1}_k(0) = {0},$ (i.e., the $\hat{f}_j$ have no common components), where $\hat{f}_j$ is the homogeneous part of degree $d$ of $f_j$, $j=1,\ldots,n$. This type of map extends holomorphically to $\mathbb{P}^n$. 
Then we can define $K$ as the set of bounded orbits 
(and $J$ can  be defined as the support of a current defined appropriately in terms of the Green's function). The set $K$ is compact with $J\subset K$. We refer to \cite{BJ2000a, BJ2000b, FS1994, FS2001b, Heinemann1996, HP1994, Klimek1995,Stawiska2005,StawiskaThesis,Ueda1998} for details.

Then we can apply the techniques of this paper to obtain the following result:

\begin{theorem} \label{thm:poly-endos}
If $f$ is a polynomial endomorphism of $\Cn$, of degree at least two, which is hyperbolic on its chain recurrent set $\mathcal{R}(f)$ (i.e., $\cR$ is a hyperbolic set for $f$ and  periodic points are dense in $\cR$), then $\cR$ is computable. 
Moreover, hyperbolicity on $\cR$ of a polynomial endomorphism of $\Cn$ is a semi-decidable property.
\end{theorem}

The interested reader may wish to investigate the computability of relevant hyperbolic invariant sets across various families of dynamical systems (both real and complex) by adapting the techniques developed in this article together with those in \cite{BoydWolf-1Dim}. The latter employs purely dynamical methods to re-establish polynomial-time computability of the Julia set and the semi-decidability of hyperbolicity for rational maps of $\CC$ (see \cite{Braverman2005} for the original proof based on complex-analytic methods).
Similarly, \cite{BoydWolf-Skew1} employed dynamical techniques to establish computability of $J$ and semi-decidability of hyperbolicity for polynomial skew products (endomorphisms) of $\Ct$. In contrast to the present article, where the invariant sets considered are attractors or saddle sets, the works \cite{BoydWolf-Skew1, BoydWolf-1Dim} primarily focus on uniformly expanding sets. Taken together with this article, these three articles provide a broad collection of tools for studying computability for different types of hyperbolic invariant sets.

\bibliographystyle{plain}

\end{document}